\documentclass[reqno]{amsart}
\usepackage{amsmath,amssymb}
\usepackage{color}
\usepackage{enumitem}
\usepackage[nospace]{cite}

\usepackage[margin={2.8cm,3.5cm}]{geometry}

\newcommand{\real}{\mathbb{R}}
\newcommand{\nat}{\mathbb{N}}

\newcommand{\rn}{\real^N}
\newcommand{\intrn}{\int_{\real^N}}

\newcommand{\N}{\mathbb{N}}
\newcommand{\R}{\mathbb{R}}

\newcommand{\ffi}{\varphi}

\newcommand{\D}{\Delta}
\newcommand{\lam}{\lambda}

\newcommand{\wto}{\rightharpoonup}

\newcommand{\wt}{\widetilde}
\newcommand{\diff}{\,\mathrm{d}}

\newcommand{\sm}{\setminus}
\newcommand{\x}{\times}

\let\le=\leqslant
\let\ge= \geqslant

\newcommand{\disp}{\displaystyle}
\newcommand\txt{\textstyle}

\DeclareMathOperator \im{Im}
\DeclareMathOperator \re{Re}
\DeclareMathOperator \dist{dist}
\DeclareMathOperator \supp{supp}

\theoremstyle{plain}
\newtheorem{theorem}{Theorem}
\newtheorem{lemma}{Lemma}
\newtheorem{proposition}{Proposition}
\newtheorem{corollary}{Corollary}
\theoremstyle{definition}
\newtheorem{definition}{Definition}
\newtheorem{remark}{Remark}

\numberwithin{equation}{section}

\newenvironment{proofof}
{\medskip\noindent{\it Proof of}}{\nolinebreak\hfill$\Box$\medskip}

\begin{document}

\title[Blow-up solutions of NLS with inverse-square potential]
{Minimal mass blow-up solutions
for the $L^2$ critical NLS with inverse-square potential}

\author[E. Csobo]{Elek Csobo}
\address{Delft Institute of Applied Mathematics \\
Delft University of Technology\\
Mekelweg~4\\
2628~CD Delft, The Netherlands}
\email{E.Csobo@tudelft.nl}

\author[F. Genoud]{Fran\c cois Genoud}
\address{Delft Institute of Applied Mathematics \\
Delft University of Technology\\
Mekelweg~4\\
2628~CD Delft, The Netherlands}
\email{S.F.Genoud@tudelft.nl}



\thanks{The authors are grateful to Dorothee Frey, Mark Veraar and Nikolaos Zographopoulos  
for helpful discussions.}


\begin{abstract}
We study minimal mass blow-up solutions of the focusing $L^2$
critical nonlinear Schr\"odinger equation with inverse-square potential,
\[
i\partial_t u + \Delta u + \frac{c}{|x|^2}u+|u|^{\frac{4}{N}}u = 0,
\]
with $N\ge 3$ and $0<c<\frac{(N-2)^2}{4}$. We first prove a sharp global well-posedness 
result: all $H^1$ solutions with a mass ({\em i.e.}~$L^2$ norm) strictly below that of the ground states
are global. Note that, unlike the equation in free space, we do not know if the ground state
is unique in the presence of the inverse-square potential. Nevertheless, all ground states
have the same, minimal, mass. We then construct and classify finite time blow-up solutions at the 
minimal mass threshold. Up to the symmetries of the equation, every such solution is a 
pseudo-conformal transformation of a ground state solution.
\end{abstract}

\maketitle


\section{Introduction}

In this paper, we initiate the study of finite time blow-up solutions of the
focusing nonlinear Schr\"odinger equation (NLS) with an attractive 
inverse-square potential,
\begin{align}\label{inls}
i\partial_t u + \Delta u + \frac{c}{|x|^2}u+|u|^{p-1}u &= 0, \quad
u(0,\cdot)=u_0\in H^1(\rn),
\end{align}
in the $L^2$ critical case, $p=1+\frac{4}{N}$, with $N\ge 3$. 
We shall fix the coupling constant 
$c\in(0,c_*)$, where $c_*=(N-2)^2/4$ is the best constant in Hardy's inequality:
\begin{equation}\label{hardy}
c_*\intrn \frac{|u|^2}{|x|^2}dx \le \intrn |\nabla u|^2 dx, 
\quad u\in H^1(\rn).
\end{equation}

The NLS equation with inverse-square potential has received substantial 
attention recently, see {\em e.g.}~\cite{visan0,visan1,visan2,suzuki,zhang_zheng} for
various results of local/global well-posedness, scattering, and harmonic analysis issues 
related to the operator $-\Delta-c|x|^{-2}$. 
All these recent contributions rely on the Strichartz estimates for this operator, which
were established by Burq {\em et al.}~in \cite{burq}.
A scattering/blow-up dichotomy result \`a la Duyckaerts--Holmer--Roudenko \cite{DHR,HR} was
proved by Killip~{\em et al.}~in \cite{visan2} for the cubic nonlinearity in dimension $N=3$, and 
recently extended by Lu~{\em et al.}~\cite{lu} to all $L^2$ supercritical, energy subcritical
nonlinearities, in dimensions $3\le N \le 6$. To the best of our knowledge, 
apart from these contributions,
blow-up solutions of \eqref{inls} are mostly virgin territory.

The present work is a first step in this direction, and we shall focus here on the 
$L^2$ critical power $p=1+\frac{4}{N}$, which is the smallest power for which finite time blow-up
occurs. $L^2$ criticality (discussed in more detail below) follows from the fact that the
potential $|x|^{-2}$ is homogeneous of degree $-2$, like the Laplace operator. 
On the other hand, the presence of a space-dependent coefficient in \eqref{inls}
breaks the translation invariance, which is a fundamental feature of the classical NLS 
({\em i.e.}~the case $c=0$). 
Mathematically, the inverse-square potential, with its remarkable scaling property, 
yields a fairly tractable instance of NLS without translation invariance. 
It also plays an important role in various areas of physics, for instance in
quantum field equations, or in the study of certain black hole solutions of the Einstein
equations; see the references in \cite{kalf,burq}.

Let us now describe the main results of our work, and their relations to the literature.
We consider strong solutions $u=u(t,x)\in C^0_t H^1_x ([0,T)\x\rn)$,
where $T>0$ is the maximum time of existence of $u$.
We will sometimes simply denote by $u(t)\in H^1(\rn)$ the function $x\mapsto u(t,x)$.
Along the flow of~\eqref{inls}, we have conservation of the $L^2$ norm, 
also known as the \emph{mass}:
\begin{equation}\label{mass}
\Vert u(t)\Vert_{L^2_x} \equiv \Vert u_0\Vert_{L^2_x},
\end{equation}
and of the \emph{energy}:
\begin{equation}\label{energy}
E(u(t))=\frac12 \intrn|\nabla u(t,x)|^2\diff x - \frac{c}{2}\int_{\rn}\frac{|u(t,x)|^2}{|x|^2}dx 
- \frac{1}{p+1}\intrn|u(t,x)|^{p+1}\diff x \equiv E(u_0).
\end{equation}
A solution is called global if $T=+\infty$. The local well-posedness
of \eqref{inls} with $c\in(0,c_*)$ is ensured by the following result. It can be proved
using Strichartz estimates as in the case of the wave equation with inverse-square
potential considered by Planchon~{\em et al.}~\cite{planchon}, although another proof is given in
\cite{suzuki}. We will comment further on this point and on the case $c=c_*$
in Subsection~\ref{threshold.sec}.

\begin{theorem}[Theorem~5.1 of \cite{suzuki}]\label{local.thm}
Let $c\in(0,c_*)$ and $1<p<1+\frac{4}{N-2}$.
For any initial value $u_0\in H^1(\rn)$, there exists $T\in(0,+\infty]$ and a maximal solution
$u\in C^0_t H^1_x ([0,T)\x\rn)$ of \eqref{inls},
satisfying \eqref{mass}--\eqref{energy} for all $t\in(0,T)$. Moreover, the blow-up 
alternative holds: if $T<+\infty$ then $\lim_{t\uparrow T} \Vert \nabla u(t)\Vert_{L^2}=+\infty$.
Finally, if $1<p<1+\frac{4}{N}$, then the solution is global.
\end {theorem}

The constants of the motion \eqref{mass}--\eqref{energy} are related 
to the symmetries of~\eqref{inls} in $H^1(\rn)$.
More precisely, if $u(t,x)$ solves~\eqref{inls}, then so do:
\begin{enumerate}[label=(\alph*)]
\item $u_{t_0}(t,x)=u(t-t_0,x)$, for all $t_0\in\R$ (time translation invariance);
\item $u_{\gamma_0}(t,x)=e^{i\gamma_0}u(t,x)$, for all $\gamma_0\in\R$ (phase invariance);
\item $u_{\lambda_0}(t,x)=\lambda_0^{2/(p-1)}u(\lambda_0^2t,\lambda_0x)$, 
for all $\lambda_0>0$ (scaling invariance).
\end{enumerate}
Note that \eqref{inls} with $c>0$ is not invariant under space translations and Galilean transformations.
The symmetries (a) and (b) are obvious and give rise, via Noether's theorem,
to the invariance of the energy and the mass, respectively. The scaling invariance is described 
in more detail below.

For our analysis it is convenient to introduce the {\em Hardy functional}, 
defined on $H^1(\rn)$ by
\[
H(u)=\int_{\rn} |\nabla u|^2dx -c\int_{\rn}\frac{|u|^2}{|x|^2}dx.
\]
Using $H$, the energy can be rewritten as
\[
E(u)=\frac12 H(u) - \frac{1}{p+1}\Vert u\Vert_{L^{p+1}}^{p+1}.
\]
Moreover by Hardy's inequality, for all $c\in (0,c_*]$,
\begin{equation}\label{hardy2}
\left(1-\frac{c}{c_*}\right) \int_{\rn} |\nabla u|^2dx\le H(u)\le \int_{\rn} |\nabla u|^2dx.
\end{equation}
In particular, for $c\in(0,c_*)$, $H(u)$ defines on $H^1(\rn)$ a seminorm equivalent to 
$\Vert\nabla u\Vert_{L^2}$. A solution $u(t)$ therefore blows up at time $T>0$ if and only if
$\lim_{t\uparrow T} H(u(t))=+\infty$. Furthermore, $H(u)$ scales as $\Vert\nabla u\Vert_{L^2}^2$ 
under space dilations. More precisely, the self-adjoint operator $-\Delta-c|x|^{-2}$ 
associated with the positive semi-definite quadratic form $H(u)$ is homogeneous 
of degree $-2$.\footnote{Note that the self-adjoint operator associated with 
$H(u)$ is unique when $c\le c_*-1$, as $-\Delta-c|x|^{-2}$ is essentially self-adjoint 
on $C_0^\infty(\rn\sm\{0\})$ in this case. For $c_*-1<c\le c_*$, this operator 
has deficiency indices $(1,1)$ and so admits a one-parameter
family of self-adjoint extensions in $L^2(\rn)$. See \cite{visan0,kalf} for more details.} 
The scaling symmetry (c) above is a crucial consequence of this fact. Now, 
$p=1+\frac{4}{N}$ yields $2/(p-1)=N/2$ and, as in the classical case $c=0$, \eqref{inls} is
invariant under the {\em $L^2$ scaling}
\[
u(t,x) \to u_\lambda(t,x)=\lambda^{N/2}u(\lambda^2 t, \lambda x) \quad (\lambda>0).
\]
This transformation preserves the $L^2$ norm and \eqref{inls} is called {\em $L^2$ critical}.

An important feature of~\eqref{inls} is the existence of standing wave solutions.
Indeed, $u(t,x)=e^{it}\ffi(x)$ is a (global) solution of~\eqref{inls} if and only if
$\ffi\in H^1(\rn)$ solves the nonlinear elliptic equation
\begin{equation}\label{sinls}
\Delta\ffi + \frac{c}{|x|^2}\ffi - \ffi + |\ffi|^{\frac{4}{N}}\ffi=0.
\end{equation}
In Section~\ref{global.sec}, we will use Weinstein's variational approach \cite{w82} to 
prove the existence in $H^1(\rn)$ of a positive radial solution $Q$ of~\eqref{sinls}, 
called {\em ground state}. Ground states will be defined as positive radial solutions of \eqref{sinls}
which minimize a suitable functional, see Proposition~\ref{minimization.prop}.
We shall see that all ground states $Q$ have the same mass 
$\Vert Q\Vert_{L^2}=:M_\mathrm{gs}>0$ and satisfy $E(Q)=0$.
We denote the set of ground states by $\mathcal{G}$.

For $c=0$, it is well known that the ground state is unique, up to the symmetries of \eqref{sinls}. 
More precisely, there exists a unique positive radial solution $Q\in H^1(\rn)$ of \eqref{sinls}, and
$\mathcal{G}=\{e^{i\gamma_0}Q(\cdot - x_0) : \gamma_0\in\real, \ x_0\in\rn\}$.
Unfortunately, for $c>0$, we are not aware of any uniqueness result for \eqref{sinls} on $\rn$. 
Uniqueness results for radial solutions of nonlinear elliptic PDEs are typically based 
on an intricate analysis of the corresponding ODEs in the radial variable $r=|x|$, 
see {\it e.g.}~\cite{kwong,coffman,peletier,yanagida}. We shall not consider this problem here.

Our first result shows that ground states play a pivotal role in the global dynamics of~\eqref{inls}. 

\begin{theorem}\label{global.thm}
Let $c\in(0,c_*)$ and $p=1+\frac{4}{N}$.
If $u_0\in H^1(\rn)$ satisfies
\begin{equation}\label{sharp}
\Vert u_0\Vert_{L^2}<M_\mathrm{gs},
\end {equation}
then the corresponding solution of \eqref{inls} given by Theorem~\ref{local.thm} is global.
\end{theorem}

The proof of this theorem relies on the inequality
\begin{equation}\label{posE}
E(u)\ge \frac12 H(u)
\left(1-\left(\frac{\Vert u\Vert_{L^2}}{M_\mathrm{gs}}\right)^\frac{4}{N}\right),
\quad u\in H^1(\rn),
\end{equation}
which follows from a sharp Gagliardo--Nirenberg inequality established in Section~\ref{global.sec}.
Indeed, since the $L^2$ norm and the energy are conserved, \eqref{posE} immediately yields
an \emph{a priori} bound on $H(u(t))$ --- and hence on $\Vert\nabla u(t)\Vert_{L^2}$ --- 
in the case $\Vert u_0\Vert_{L^2}<M_\mathrm{gs}$. Namely, we have
\begin{equation} \label{subcritical}
H(u(t)) \le 2E(u_0) \left(1-\left(\frac{\Vert u\Vert_{L^2}}{M_\mathrm{gs}}\right)^\frac{4}{N}\right)^{-1},
\end{equation}
which implies global existence. 

We shall next exhibit blow-up solutions at the mass threshold 
$\Vert u_0\Vert_{L^2}=M_\mathrm{gs}$, which is thus
the {\em minimal mass} where blow-up can occur. This shows that the global
well-posedness condition \eqref{sharp} is sharp.
The minimal mass blow-up solutions are constructed explicitly by applying the pseudo-conformal
transformation (defined in Lemma~\ref{pseudo-conf.lem}) to the standing wave $e^{it}Q$. 
Taking into account the above symmetries of~\eqref{inls} we obtain, for each
ground state $Q\in \mathcal{G}$, a 3-parameter
family $(S_{Q,T,\lambda_0,\gamma_0})_{T\in\R,\lambda_0>0,\gamma_0\in\R}$ of minimal 
mass solutions of~\eqref{inls} blowing up in finite time, defined as
\begin{equation} \label{S1.def}
S_{Q,T,\lambda_0,\gamma_0}(t,x) = e^{i\gamma_0} e^{i\frac{\lambda_0^2}{T-t}} e^{-i\frac{|x|^2}{4(T-t)}}
\left(\frac{\lambda_0}{T-t}\right)^{N/2}Q\left( \frac{\lambda_0x}{T-t} \right).
\end{equation}
Note that these solutions present a self-similar profile, in the sense that, for all $t\in [0,T)$, there
exists $\lam(t)>0$ such that $|S_{Q,T,\lambda_0,\gamma_0}(t,x)|=\lam(t)^{N/2}Q(\lam(t)x)$. Hence,
up to a time-dependent $L^2$ rescaling, $S_{Q,T,\lambda_0,\gamma_0}$ 
keeps the same shape as $Q$ while blowing up.

The striking fact is that all finite time blow-up solutions at the minimal mass threshold
are of this form. Indeed, we have the following classification result.

\begin{theorem} \label{main.thm}
Let $T>0$ and $u\in C\big([0,T),H^1(\rn)\big)$ be a minimal mass solution of~\eqref{inls}
with $p=1+\frac{4}{N}$, which blows up at time $T$, 
\emph{i.e.}~$\Vert u_0\Vert_{L^2}=M_\mathrm{gs}$ and $\lim_{t\uparrow T} H(u(t))=+\infty$.
Then there exist a ground state $Q\in \mathcal{G}$, $\lambda_0>0$ 
and $\gamma_0\in\R$ such that, for all $t\in [0,T)$,
\[
u(t)=S_{Q,T,\lambda_0,\gamma_0}(t).
\]
\end{theorem}

Blow-up solutions of the $L^2$ critical NLS in the classical case $c=0$ 
have been thoroughly investigated in the past.
A comprehensive review of the theory  
for the classical focusing NLS can be found in~\cite{r}.
Theorems~\ref{global.thm} and \ref{main.thm} respectively
extend the famous results of Weinstein~\cite{w82} and Merle~\cite{merle}, 
from the case $c=0$ to the case $c\in(0,c_*)$. We are indebted to these authors for the
fundamental ideas supporting our proofs. In fact, our approach here is based closely on 
Hmidi and Keraani \cite{keraani}, where the arguments of \cite{merle} have been simplified, 
using a Cauchy--Schwarz inequality due to Banica \cite{banica}.

The paper \cite{keraani} deals with the
classical NLS with $L^2$ critical nonlinearity and constant coefficients. To adapt it to
space-dependent coefficients, the main difficulties lie in a crucial compactness result.
In the present work, this is Proposition~\ref{compact.prop}, which
relies on a subtle combination of Hardy's inequality and the sharp Gagliardo--Nirenberg inequality
\eqref{GN}. 
Various authors have also considered blow-up solutions for focusing NLS equations 
with space-dependent coefficients, notably \cite{rs,bcd,m96}. 
These references are discussed in some detail in the introduction of the paper 
\cite{combet} by Combet and the second author, where the classification 
of minimal mass blow-up solutions is obtained for the $L^2$ critical equation
\[
i\partial_t u + \Delta u + |x|^{-b}|u|^{\frac{4-2b}{N}}u = 0, \quad 
\text{with} \ 0<b<\min\{2,N\}, \ N\ge1.
\]


\subsection{The threshold case $c=c_*$}\label{threshold.sec}

Our initial motivation for studying \eqref{inls} came from the paper
\cite{greeks} by Trachanas and Zographopoulos, 
where the orbital stability of standing waves of \eqref{inls} is considered, with 
a special focus on the threshold value $c=c_*$. From a functional analytic perspective, an
interesting difficulty arises in this case due to the sharpness of Hardy's inequality at $c=c_*$.
The natural energy space associated with \eqref{inls},
$\mathcal{H}=\{u\in L^2(\rn) : H(u)<+\infty\}$, then satisfies $H^1(\rn)\subsetneq\mathcal{H}$. Indeed,
the ground states of \eqref{inls} with $c=c_*$ have a singularity of order $|x|^{-(N-2)/2}$ at the origin,
and thus lie in $\mathcal{H}\sm H^1(\rn)$. Fine properties of the space $\mathcal{H}$ are used
in \cite{greeks} to carry out a variational analysis of orbital stability, and our initial hope was to
be able to extend our analysis to the case $c=c_*$ using this functional framework.
However, it is not clear to us that local well-posedness holds in this case. 
Trachanas and Zographopoulos \cite[Theorem~3.1]{greeks} 
claim that (in the radial case) it follows by adapting a proof by Cazenave,
but we were not able to carry this through. Let us briefly explain why. 

For $c\in(0,c_*)$, Theorem~\ref{local.thm} above was proved in \cite{suzuki}, 
by adapting to $-\Delta-c|x|^{-2}$ Cazenave's proof of \cite[Theorems~3.3.9]{cazenave}, 
originally developed to deal with $-\Delta$ in bounded domains, 
where dispersive estimates are not available. 
This approach allows one to obtain existence of local (in time) solutions \cite[Theorems~3.3.5]{cazenave}, 
but an additional uniqueness result is required to obtain the full well-posedness result
\cite[Theorems~3.3.9]{cazenave}. In \cite{suzuki}, uniqueness relies on the Strichartz
estimates for $-\Delta-c|x|^{-2}$, which were established in \cite{burq}.
Unfortunately, as pointed out on p.~521 of \cite{burq}, 
these estimates break down at the threshold value $c=c_*$. 
Hence, the existence of local in time solutions is
ensured by \cite{suzuki}, but it is not clear if and how uniqueness can be proved.
As uniqueness is essential in our proof of Theorem~\ref{main.thm},
we shall only consider $c\in(0,c_*)$ here. Note that, for $c=c_*$, inequality \eqref{crucial} also
breaks down, and we do not know how to prove the crucial Proposition~\ref{compact.prop}.


\section{Ground states and the sharp global existence criterion}\label{global.sec}

In this section, we will prove Theorem~\ref{global.thm}.
We start by solving a minimization problem, the minimum of which is
attained at the ground states of the stationary equation. 
A crucial consequence will be the sharp Gagliardo--Nirenberg inequality leading to \eqref{posE}.
As these results may be useful in other problems, we will state them for any $1<p<1+\frac{4}{N-2}$. 
A similar problem was considered in \cite{visan2} for the specific case of three space dimensions and $p=3$.
Consider the Weinstein functional
\[
J^{p,N}(u):=\frac{H(u)^{\frac{p-1}{4}N}\|u\|^{2+\frac{p-1}{2}(2-N)}_{L^2}}{\|u\|^{p+1}_{L^{p+1}}}.
\]
 
\begin{proposition}\label{minimization.prop}
 For $1<p<1+\frac{4}{N-2}$,
\[
\alpha_{p,N} := \inf_{u\in H^1(\rn)\backslash\{0\}} J^{p,N}(u)
\]
is attained at a positive radial function $Q\in H^1(\rn)$, solution of the Euler--Lagrange equation
\begin{equation}\label{EL}
N\Big(\frac{p-1}{4}\Big)\Big(\Delta Q +c\frac{Q}{|x|^2}\Big)-\left(1+\frac{p-1}{4}(2-N)\right)Q+Q^p=0.
\end{equation}
Furthermore,
\begin{equation} \label{sharpestimate}
\alpha_{p,N}=2\frac{\|Q\|_{L^2}^{p-1}}{p+1}.
\end{equation}
In the case $p=1+\frac4N$, any minimizer of $J^{p,N}$ can be rescaled into a solution of \eqref{sinls}.
\end{proposition}


\begin{proof}
First note that the functional $J^{p,N}$ is invariant under the scaling 
\[
u(x) \to u^{\lambda,\mu}(x):=\mu u(\lambda x) \quad (\lambda,\mu>0).
\]
Indeed, we have
\begin{align*}
H (u^{\lambda,\mu})&=\lambda^{2-N}\mu^2 H (u),\\
\|u^{\lambda,\mu}\|_{L^2}^2&=\lambda^{-N}\mu^2\|u\|^2_{L^2},\\
\|u^{\lambda,\mu}\|^{p+1}_{L^{p+1}}&=\lambda^{-N}\mu^{p+1}\|u\|^{p+1}_{L^{p+1}}\\
\text{and so} \quad J^{p,N}(u^{\lambda,\mu})&=J^{p,N}(u).
\end{align*}
Let $\{u_n\}\subset H^1(\rn)$ be a minimizing sequence,
$\alpha_{p,N} = \lim_{n\rightarrow\infty} J^{p,N}(u_n)\ge0$. 
Since $J(|u|)\le J(u)$, we can suppose that $u_n\ge0$ for all $n\in\N$. Furthermore,
denoting by $u_n^*$ the Schwarz symmetrization of $u_n$ (see e.g.~\cite[pp.~80-83]{Lieb}),
we have that
\[
\intrn \frac{(u_n^*)^{2}}{|x|^2} dx \ge \intrn \frac{u_n^{2}}{|x|^2} dx,
\]
\[
\Vert \nabla u_n^*\Vert_{L^2}\le \Vert \nabla u_n\Vert_{L^2}
\quad\text{and}\quad
\Vert u_n^*\Vert_{L^2}\le\Vert u_n\Vert_{L^2}.
\]
We can therefore suppose that $u_n=u_n^*$ and, in particular, that each $u_n$ is positive,
radial and radially decreasing.
Thanks to the scaling invariance of $J^{p,N}$, 
we can further rescale the minimizing sequence by choosing 
$\lambda_n=\|u_n\|_{L^2}/H(u_n) $ and $\mu_n=\|u_n\|_{L^2}^{N/2-1}/H^{N/2}(u_n)$. 
We thus obtain a minimizing sequence $\psi_n=u_n^{\lambda_n,\mu_n}$ 
with the following properties:
\begin{align*}
\psi_n \ge 0, \quad \psi_n=\psi_n(|x|),\\
\|\psi_n\|_{L^2}=1, \quad H(\psi_n)=1,\\
\lim_{n\rightarrow\infty} J^{p,N}(\psi_n)=\alpha_{p,N}.
\end{align*}
In particular, $\{\psi_n\}$ is bounded in $H^1(\rn)$.
Thus, up to a subsequence, we can suppose that $\{\psi_n\}$ has a weak limit 
$\psi^* \in H^1(\rn)$.
Since $\{\psi_n\}\subset H^1_\mathrm{rad}(\rn)$ and $p+1\in(2,2^*)$, we can also suppose that
$\psi_n$ converges strongly to $\psi^*$ in $L^{p+1}(\rn)$. 
By weak lower semi-continuity of $\Vert\cdot\Vert_{L^2}$ and $H$ (see \cite{montefusco}), 
we obtain that $\|\psi^*\|_{L^2}\le 1$ and $H(\psi^*)\le 1$. Hence,
\[
\alpha_{p,N}\le J^{p,N}(\psi^*)\le \frac{1}{\|\psi^*\|^{p+1}_{L^{p+1}}}=\lim_{n\rightarrow\infty}J^{p,N}(\psi_n)=\alpha_{p,N}.
\]
It follows that $H(\psi^*)^{\frac{p-1}{4}N}\|\psi^*\|_{L^2}^{2+(p-1)(2-N)/2}=1$ and, 
therefore, $H(\psi^*)=\|\psi^*\|_{L^2}=1$.

Now, $\psi^*$ must satisfy the Euler--Lagrange equation $DJ^{p,N}(\psi^*)=0$.
We first remark that, for any $u \in H^1(\rn)$, $DJ^{p,N}(u)=0$ reads
\begin{equation}\label{euler}
-N\Big(\frac{p-1}{4}\Big)\Big(\Delta u + \frac{c}{|x|^2} u\Big)
+\Big(1+ \frac{p-1}{4}(2-N)\Big)\frac{H(u)}{\|u\|_{L^2}^2} u
-\Big(\frac{p+1}{2}\Big) \frac{H(u)}{\|u\|_{L^{p+1}}^{p+1}}|u|^{p-1}u=0.
\end{equation}
Taking into consideration $H(\psi^*)=1$, $\| \psi^*\|_{L^2}=1$, 
$\|\psi^*\|^{p+1}_{L^{p+1}}=\alpha_{p,N}^{-1}$ and $\psi^*\ge0$, we get
\[
N\Big(\frac{p-1}{4}\Big)\Big(\Delta \psi^*+c\frac{\psi^*}{|x|^2}\Big)
-\left(1+\frac{p-1}{4}(2-N) \right)\psi^*+\alpha_{p,N}\Big(\frac{p+1}{2}\Big)(\psi^*)^p=0.
\]
Then $Q$ defined by $\psi^*=(\alpha_{p,N}(p+1)/2)^{-1/(p-1)}Q$ readily satisfies (\ref{EL})
and \eqref{sharpestimate}.

Finally, in view of Remark~\ref{gs_scaling.rem} and Remark~\ref{scaling.rem} below,  
if $p=1+\frac4N$ and $u$ is a minimizer of $J^{1+\frac{4}{N},N}$,
the Euler--Lagrange equation \eqref{euler} reads
\begin{equation}\label{euler2}
\Delta u + \frac{c}{|x|^2} u
-\frac{2}{N}\frac{H(u)}{\|u\|_{L^2}^2} u +|u|^{p-1}u=0,
\end{equation}
and any solution $u$
can be rescaled into a solution $Q$ of \eqref{sinls} by letting
\[
u(x)=\lambda^{N/2}Q(\lambda x) \quad\text{with}\quad 
\lambda=\sqrt{\frac{2}{N}}\frac{\sqrt{H(u)}}{\Vert u\Vert_{L^2}}.
\]
The proof is complete.
\end{proof}

\begin{remark}\label{gs_scaling.rem}
\noindent
(a) Note that, by phase invariance of $J^{p,N}$, if $u$ is a minimizer so is 
$e^{i\theta}u, \ \theta\in\R$.

\medskip
\noindent
(b) For $p=1+\frac{4}{N}$, the Pohozaev identity for \eqref{sinls} implies
that any minimizer $u$ of $J^{1+\frac{4}{N},N}$ satisfies $E(u)=0$, that is 
\[
H(u)=\frac{1}{1+\frac{2}{N}}\|u\|_{L^{2+\frac{4}{N}}}^{2+\frac{4}{N}},
\]
and so
\begin{equation}\label{valueofmin}
J^{1+\frac{4}{N},N}(u)
=\frac{H(u)\Vert u\Vert_{L^2}^{\frac{4}{N}}}{\|u\|_{L^{2+\frac{4}{N}}}^{2+\frac{4}{N}}}
=\frac{\Vert u\Vert_{L^2}^{\frac{4}{N}}}{1+\frac{2}{N}}.
\end{equation}
Therefore, all minimizers have the same mass. 

\medskip
\noindent (c) The proof of Proposition~\ref{charact.prop} below shows
that all minimizers are positive and radial, up to a global phase factor.
\end{remark}

\begin{definition}\label{gs.def}
In the case $p=1+\frac{4}{N}$,
we call {\em ground states} the minimizers of $J^{1+\frac{4}{N},N}$ 
that are positive radial solutions of \eqref{sinls}. We denote the set of ground states by $\mathcal{G}$.
In view of Remark~\ref{gs_scaling.rem}~(b), there exists $M_\mathrm{gs}>0$ such that
$\Vert Q\Vert_{L^2}=M_\mathrm{gs}$ for all $Q\in\mathcal{G}$.
We call $M_\mathrm{gs}$ the {\em minimal mass}.
\end{definition}

\begin{remark}\label{scaling.rem}
The $L^2$ scaling $u_\lambda(x)=\lambda^{N/2}u(\lambda{x})$ yields
\[
H(u_\lambda)=\lambda^2H(u), \quad \Vert u_\lambda\Vert_{L^2}=\Vert u\Vert_{L^2},
\quad \Vert u_\lambda\Vert_{L^{p+1}}^{p+1}=\lambda^2 \Vert u\Vert_{L^{p+1}}^{p+1},
\quad \text{hence} \quad J^{p,N}(u_\lambda)=J^{p,N}(u).
\]
Therefore, if $Q$ is a ground state, then $Q_\lambda(x)=\lambda^{N/2}Q(\lambda x)$ 
is again a minimizer of $J^{1+\frac{4}{N},N}$, satisfying $\Vert Q_\lambda\Vert_{L^2}=M_\mathrm{gs}$ 
and the rescaled Euler--Lagrange equation
\[
\Delta Q_\lambda+ \frac{c}{|x|^2}Q_\lambda-\lambda^{2} Q_\lambda 
+ |Q_\lambda|^{\frac{4}{N}}Q_\lambda=0.
\]
It follows from Remark~\ref{minmass.rem} below that the
ground states are the solutions of \eqref{sinls} with smallest mass.
\end{remark}

The following sharp Gagliardo--Nirenberg inequality is an immediate consequence 
of Proposition~\ref{minimization.prop}.

\begin{corollary}\label{GN.cor} For $1<p<1+\frac{4}{N-2}$ and $u\in H^1(\rn)$, we have
\begin{equation}\label{GN}
\Vert u\Vert_{L^{p+1}}^{p+1} \le 
\alpha_{p,N}^{-1} H(u)^{\frac{p-1}{4}N} \Vert u\Vert_{L^{2}}^{2+\frac{p-1}{2}(2-N)},
\end{equation}
where
\[
\alpha_{p,N}=\frac{M_\mathrm{gs}^{\frac{4}{N}}}{1+\frac{2}{N}}
\quad\text{in the critical case} \quad p=1+\frac{4}{N}. 
\]
\end{corollary}

We are now in a position to prove our global well-posedness result.

\begin{proofof}~{\it Theorem~\ref{global.thm}}.
Consider a local solution of \eqref{inls}, as given by Theorem~\ref{local.thm}.
By the blow-up alternative, we need only show that, if $\Vert u_0\Vert_{L^2}<M_\text{gs}$,
then $H(u(t))$ remains bounded. 
We control $H(u(t))$ using the energy. By definition of $E$,
\[
H(u)=2E(u) + \frac{1}{1+\frac2N}\Vert{u}\Vert_{L^{2+\frac4N}}^{2+\frac4N}.
\]
Now, by Corollary~\ref{GN.cor},
\[
\|u\|_{L^{2+\frac4N}}^{2+\frac4N}\le 
\left(\frac{1+\frac2N}{M_\text{gs}^{\frac4N}}\right) H(u)\|u\|_{L^2}^{\frac4N}
=(1+\txt\frac2N)\disp\left(\frac{\|u\|_{L^2}}{M_\text{gs}}\right)^{\frac4N}H(u).
\]
Hence, by conservation of mass and energy,
\begin{align*}
H(u(t))	&\le 2E(u(t))+\left(\frac{\|u(t)\|_{L^2}}{M_\text{gs}}\right)^{\frac4N}H(u(t))\\
		&\phantom{\le}=2E(u_0)+\left(\frac{\|u_0\|_{L^2}}{M_\text{gs}}\right)^{\frac4N}H(u(t)).
\end{align*}
It follows that
\[
\left(1-\left(\frac{\|u_0\|_{L^2}}{M_\text{gs}}\right)^{\frac4N}\right)H(u(t))\le 2 E(u_0),
\]
which concludes the proof.
\end{proofof}


\section{Construction of minimal mass blow up solutions}\label{minmass.sec}

From now on and for the rest of the paper, unless specified otherwise, we focus on
the $L^2$ critical nonlinearity $p=1+\frac4N$.
In this section we construct finite time blow-up solutions to \eqref{inls} 
with minimal mass $M_\mathrm{gs}$. 
We thereby show that the condition for global well-posedness in Theorem~\ref{global.thm} 
is sharp. First we show that the equation is invariant under the pseudo-conformal transformation, 
which is defined in the following lemma. 

\begin{lemma}\label{pseudo-conf.lem}
Let $u$ be a global solution of \eqref{inls}. Then, for all $T\in\R$, the function
\[ 
u_T(t,x)=\frac{e^{-i\frac{|x|^2}{4(T-t)}}}{(T-t)^{N/2}}u\left(\frac{1}{T-t},\frac{x}{T-t}\right)
\]
is a solution of \eqref{inls} on $(-\infty, T)$, with $\|u_T\|_{L^2}=\|u\|_{L^2}$.
\end{lemma}

\begin{proof}
Straightforward calculations give
\[
\partial_t u_T(t,x) = 
\frac{e^{-i\frac{|x|^2}{4(T-t)}}}{(T-t)^{N/2+2}} \left[ \frac{N}{2}(T-t)u -i\frac{|x|^2}{4}u
+\partial_t u +x\cdot\nabla u \right]\! \left(\frac{1}{T-t},\frac{x}{T-t}\right)
\]
and
\[
\D u_T(t,x) = \frac{e^{-i\frac{|x|^2}{4(T-t)}}}{(T-t)^{N/2+2}} \left[ -\frac{|x|^2}{4}u -i\frac{N}{2}(T-t)u
-ix\cdot\nabla u +\D u\right]\! \left(\frac{1}{T-t},\frac{x}{T-t}\right)
\]
for the derivatives.
For the nonlinear term we find
\[
|u_T|^{\frac{4}{N}}u_T(t,x)=
\frac{e^{-i\frac{|x|^2}{4(T-t)}}}{(T-t)^{N/2+2}}|u|^{\frac{4}{N}}u\Big(\frac{1}{T-t},\frac{x}{T-t}\Big),
\]
and for the potential term
\[
c|x|^{-2}u_T(t,x)=
c\Big|\frac{x}{T-t}\Big|^{-2}\frac{e^{-i\frac{|x|^2}{4(T-t)}}}{(T-t)^{N/2+2}}u\Big(\frac{1}{T-t},\frac{x}{T-t}\Big).
\]
It then follows from \eqref{inls} that
\[
\begin{split}
i\partial u_T(t,x)+&\Delta u_T(t,x)+c|x|^{-2}u_T(t,x)+|u_T|^{4/N}u_T(t,x)
\\
&=\frac{e^{-i\frac{|x|^2}{4(T-t)}}}{(T-t)^{N/2+2}}[i\partial_tu+\Delta u+c|x|^{-2}u
+|u|^{4/N}u]\Big(\frac{1}{T-t},\frac{x}{T-t}\Big)=0,
\end{split}
\]
which proves the lemma.
\end{proof}

\begin{remark}\label{minmass.rem}
It follows from Lemma~\ref{pseudo-conf.lem} that any solution $\ffi$ of \eqref{sinls}
satisfies $\|\ffi\|_{L^2}\ge M_\mathrm{gs}$. Otherwise, applying Lemma~\ref{pseudo-conf.lem}
with $u(t,x)=e^{it}\ffi(x)$, one could construct a finite time blow-up solution below the minimal mass 
threshold, which would contradict our global well-posedness result, Theorem~\ref{global.thm}.
\end{remark}

We now use the pseudo-conformal transformation and the symmetries of the equation 
to construct explicit blow-up solutions.

\begin{proposition} \label{blow-up_sol.prop}
For all $Q\in\mathcal{G}$, $T\in\R$, $\lambda_0>0$ and $\gamma_0\in\R$, the function 
$S_{Q,T,\lambda_0,\gamma_0}$, defined by
\begin{equation} \label{S2.def}
S_{Q,T,\lambda_0,\gamma_0}(t,x) = e^{i\gamma_0} e^{i\frac{\lambda_0^2}{T-t}} e^{-i\frac{|x|^2}{4(T-t)}}
\left(\frac{\lambda_0}{T-t}\right)^{N/2}Q\left( \frac{\lambda_0x}{T-t} \right),
\end{equation}
is a minimal mass solution of~\eqref{inls} defined on $(-\infty,T)$, and which blows up with speed
\[
\Vert\nabla S_{Q,T,\lambda_0,\gamma_0}(t)\Vert_{L^2} \sim \frac{C}{T-t} \quad \text{as } t\to T,
\]
for some $C>0$.
\end{proposition}

\begin{proof}
The proposition is a simple consequence of Lemma~\ref{pseudo-conf.lem} applied to the global solution
\[
u_{Q,\lambda_0,\gamma_0}(t,x) = e^{i\gamma_0} e^{i\lambda_0^2 t}\lambda_0^{N/2}Q(\lambda_0x),
\]
which is given by applying the scaling and phase symmetries to the standing wave $u(t,x)=e^{it}Q(x)$.
\end{proof}

\begin{remark}
Note that the blow-up solutions of the family exhibited in Proposition~\ref{blow-up_sol.prop} 
can all be retrieved from the solution
\[
S(t,x) := S_{0,1,0}(t,x) = e^{i\frac{|x|^2}{4t}} e^{-\frac{i}{t}} \frac{1}{|t|^{N/2}}Q\left(-\frac{x}{t}\right),
\]
defined on $(-\infty,0)$ and which blows up at $t=0$ with speed
\[
\Vert\nabla S(t)\Vert_{L^2} \sim \frac{C}{|t|}\quad \text{as } t\uparrow0,
\]
for some $C>0$. Indeed, all the solutions $S_{Q,T,\lambda_0,\gamma_0}$ are equal to $S$,
up to the symmetries (a), (b) and (c) stated in the introduction.
Namely, if we apply the changes $u(t,x)\to \lambda_0^{-N/2}u(\lambda_0^{-2}t,\lambda_0^{-1}x)$, 
$u(t,x)\to u(t-T,x)$
and finally $u(t,x)\to e^{i\gamma_0}u(t,x)$ to $S$, we obtain $S_{Q,T,\lambda_0,\gamma_0}$.
\end{remark}


\section{Classification of minimal mass blow-up solutions}
In this section we prove the classification of minimal mass blow-up solutions of \eqref{inls}
for the $L^2$ critical nonlinearity $p=1+\frac4N$.
We start with a variational result, characterizing ground state solutions.

\subsection{Variational characterization of ground states}

\begin{proposition}\label{charact.prop}
Let $v\in H^1(\rn)$ be such that 
\begin{equation} \label{varchar}
\|v\|_{L^2}=M_\mathrm{gs} \quad \text{and} \quad E(v)=0.
\end{equation}
Then there exist $Q\in\mathcal{G}$, and $\gamma_0 \in \R$ 
such that $v(x)=e^{i\gamma_0}\lambda_0^{N/2}Q(\lambda_0x)$,
where $\disp\lambda_0=\sqrt{\frac{2}{N}}\frac{\sqrt{H(|v|)}}{M_\mathrm{gs}}$.
\end{proposition}

\begin{proof}
It follows directly from \eqref{valueofmin} and \eqref{varchar} that $v$ is a 
minimizer of $J^{1+\frac4N,N}$. Since $J^{1+\frac4N,N}(|v|)\le J^{1+\frac4N,N}(v)$,
$|v|$ is also a minimizer. Furthermore, any positive minimizer is radial. Indeed, suppose $v_0$ is a positive minimizer that is not radial,
and let $v_0^*$ be its Schwarz symmetrization. 
Then, in view of \cite[Theorem~3.4]{Lieb}, we have that
\[
\intrn \frac{|v_0^*|^{2}}{|x|^2} dx > \intrn \frac{|v_0|^{2}}{|x|^2} dx.
\]
Since, on the other hand,
\[
\Vert \nabla v_0^*\Vert_{L^2}\le \Vert \nabla v_0\Vert_{L^2}
\quad\text{and}\quad
\Vert v_0^*\Vert_{L^2}\le\Vert v_0\Vert_{L^2},
\]
we get $J^{1+\frac4N,N}(v_0^*)<J^{1+\frac4N,N}(v_0)$, a contradiction. We deduce that $|v|$ is radial. Now, 
$|v|$ satisfies the Euler--Lagrange equation \eqref{euler2} and we know
there exists $Q\in\mathcal{G}$ such that
\[
|v(x)|=\lambda_0^{N/2}Q(\lambda_0 x), \quad \text{where} \quad \lambda_0
=\Big(\frac{2}{N}\Big)^{\frac{1}{2}}\frac{H(|v|)^\frac{1}{2}}{\|v\|_{L^2}}.
\]
The proof will be complete if we show that $w(x)=v(x)/|v(x)|$ is constant on $\rn$, 
for this implies that there exists $\gamma_0\in\R$ such that 
$w(x)=e^{i\gamma_0}$ for all $x\in\rn$.
Differentiating $|w|^2=1$ yields $\re(\overline{w}\nabla w)\equiv 0$, thus
\[
|\nabla v|^2=|\nabla(|v|)|^2+|v|^2|\nabla w|^2+2|v|\nabla(|v|)\cdot \re(\overline{w}\nabla w)
\]
and
\[
\|\nabla v\|_{L^2}^2=\|\nabla (|v|)\|_{L^2}^2+\int_{\rn} |v|^2|\nabla w|^2dx.
\]
Together with Hardy's inequality, this implies that, if $|\nabla w|\neq 0$
then $J^{1+\frac4N,N}(|v|)< J^{1+\frac4N,N}(v)$, which is a contradiction. 
Hence, $w$ is constant, and the proof is complete.
\end{proof}


\subsection{Compactness} 

We now establish a compactness result which will play a crucial role 
in the proof of Theorem~\ref{main.thm}.
Owing to the equivalence of the seminorms $\|\nabla u\|_{L^2}^2$ and $H(u)$, 
the following concentration-compactness lemma can be proved 
by minor modifications to the proof of Proposition~1.7.6 in \cite{cazenave}.

\begin{lemma}\label{conc-comp.lemma}
Let $\{v_n\}_{n\in \N}\subset H^1(\rn)$ satisfy
\[
\lim_{n \rightarrow\infty}\|v_n\|_{L^2}=M<+\infty \quad \text{and} \quad \sup_{n\in\N} H(v_n) <+\infty.
\]
Then there exists a subsequence $\{v_{n_k}\}_{k\in \N}$ that satisfies one of the following properties.

\medskip
\noindent
(V) $\|v_{n_k}\|_{L^q}\rightarrow 0$ as $k\rightarrow\infty$ for all $q\in(2,2^*)$.

\medskip
\noindent
(D) There are sequences ${w_k},{z_k}\in H^1(\rn)$ and a constant $\alpha\in (0,1)$ such that:
\begin{enumerate}
\item $\dist(\supp(w_k),\supp(z_k))\rightarrow\infty$;
\item $|v_k|+|z_k|\le |v_{n_k}|$;
\item $\sup_{k\in\N} (\|w_k\|_{H^1}+\|z_k\|_{H^1})<+\infty$;
\item $\|w_k\|_{L^2}\rightarrow\alpha M$ and $\|z_k\|_{L^2}\to (1-\alpha)M$ as $k\rightarrow\infty$;
\item $\lim_{k\rightarrow\infty}\big| \int_{\rn}|v_{n_k}|^q-\int_{\rn}|w_k|^q-\int_{\rn}|z_k|^q\big|=0$ 
for all $q\in[2,2^*)$;
\item $\liminf_{k\rightarrow\infty}\big\{  H(v_{n_k})-H(w_k)-H(z_k) \big\}\ge 0$.
\end{enumerate} 

\medskip
\noindent
(C) There exists $v\in H^1(\rn)$ and a sequence $\{y_k\}_{k\in\N}\subset \rn$ such that
\[
v_{n_k}(\cdot -y_k)\rightarrow v \quad \text{in} \quad L^q(\rn), \quad \forall q\in[2,2^*).
\]
\end{lemma}

We are now in a position to prove the following proposition, which is the key step
to the classification of minimal mass blow-up solutions. 

\begin{proposition}\label{compact.prop}
Consider a sequence $\{v_n\}_{n\in\nat}\subset H^1(\rn)$ satisfying
\begin{equation} \label{a}
\lim_{n\rightarrow\infty}\|v_n\|_{L^2}=M_\mathrm{gs}, \quad 0<\limsup_{n\in\N}H(v_n)<+\infty,
\quad \limsup_{n\rightarrow\infty}E(v_n)\le 0.
\end{equation}
Then there exists a subsequence $\{v_{n_k}\}_{k\in\nat}$, $\gamma_0\in\R$ and $Q\in\mathcal{G}$ such that
\begin{equation}\label{compact}
\lim_{k\rightarrow\infty}\|v_{n_k}-e^{i\gamma_0}Q\|_{H^1}=0.
\end{equation}
\end{proposition}

\begin{proof}
The behavior of the sequence $\{v_n\}$ is constrained by the concentration-compactness lemma. 
Our proof proceeds in several steps: by ruling out (V) and (D),
we will first show that property (C) holds in Lemma~\ref{conc-comp.lemma}. We will then prove
the localization of the subsequence $\{v_{n_k}\}$, {\em i.e.}~that
the sequence of translations ${y_k}$ is bounded in $\rn$. This step is the most delicate one
due to the inverse-square potential, and will be handled by a subtle combination of Hardy's inequality
and the Gagliardo--Nirenberg inequality \eqref{GN}. Once this localization property is established, 
the result will follow from Proposition~\ref{charact.prop}.

\medskip
\noindent
{\em Step~1: Compactness in $L^q(\rn)$.}
First suppose that (V) holds. Since $p+1\in (2,2^*)$ we get,
\[
\limsup_{k\rightarrow\infty} E(v_{n_k})=\limsup_{k\rightarrow\infty}\Big\{\frac{1}{2}H(v_{n_k})-\frac{1}{p+1}\int_{\rn} |v_{n_k}|^{p+1}dx\Big\}=\frac{1}{2}\limsup_{k\rightarrow\infty}H(v_{n_k})>0,
\]
which contradicts our assumptions.

Suppose now that (D) holds. Then from properties (5) and (6) it follows that
\[
\begin{split}
\limsup_{k\rightarrow\infty}\{E(w_k)+E(z_k)\}&\le\frac{1}{2}\liminf_{k\rightarrow\infty}H(v_{n_k})
-\frac{1}{p+1}\limsup_{k\rightarrow\infty}\int_{\rn}|v_{n_k}|dx \\
&\le\limsup_{k\rightarrow\infty} E(v_{n_k})\le 0.
\end{split}
\]
Property (D)(4) of Lemma (\ref{conc-comp.lemma}) and inequality \eqref{posE} 
imply that $E(w_k), E(z_k)\ge 0$ for $k$ large enough, thus
\[
E(w_k)\rightarrow 0 \quad \text{and} \quad E(z_k)\rightarrow 0.
\]
Using again property (4) and inequality \eqref{posE}, we get that
\[
H(w_k)\rightarrow 0 \quad\text{and}\quad H(z_k)\rightarrow 0.
\]
Hence,
\[
\lim_{k\rightarrow\infty}\int_{\rn}|v_{n_k}|^{p+1}dx
=\lim_{k\rightarrow\infty}\left(\int_{\rn}|w_k|^{p+1}dx+\int_{\rn}|z_k|^{p+1}dx\right)=0,
\]
leading again to the contradiction that $\limsup_{k\rightarrow\infty}E(v_{n_k})>0$. Consequently,
by Lemma~\ref{conc-comp.lemma}, there exists $v\in H^1(\rn)$
and a sequence of translations $\{y_k\}_{k\in\N}\subset \rn$ such that
\begin{equation}\label{tildecompact}
\wt{v}_{n_k} \rightarrow v \quad \text{in} \quad L^q(\rn), \quad \forall q\in[2,2^*),
\end{equation}
where we let
\[
\wt{v}_{n_k}(x):=v_{n_k}(x-y_k). 
\]
In particular, the sequence $\{\wt{v}_{n_k}\}$
is compact in $L^2(\rn)$ and, in view of \eqref{a}, $\|v\|_{L^2}=M_\mathrm{gs}$. We shall
also suppose, without loss of generality, that $\wt{v}_{n_k}\wto v$ weakly in $H^1(\rn)$.

\medskip
\noindent
{\em Step~2: Localization.}
We now prove by contradiction that the sequence of translations $\{y_k\}$ is bounded in $\rn$. 
Let us assume that there exists a subsequence, denoted again by $\{y_k\}$, 
such that $|y_k|\rightarrow\infty$ as $k\rightarrow\infty$.
We will prove that, in this case, $\liminf_{k\rightarrow\infty}E(v_{n_k})>0$, which contradicts \eqref{a}.

We start by showing that
\begin{equation}\label{zerolimit}
\lim_{k\to\infty} \intrn \frac{|v_{n_k}|^2}{|x|^2}dx = 0.
\end{equation}
To prove \eqref{zerolimit}, we split the integral as
\[
\intrn \frac{|v_{n_k}|^2}{|x|^2}dx =
\underbrace{\int_{|x|<R}\frac{|v_{n_k}|^2}{|x|^2}dx}_{\mathrm{I}_k}
+\underbrace{\int_{|x|\ge R}\frac{|v_{n_k}|^2}{|x|^2}dx}_{\mathrm{II}_k},
\]
for some $R>0$.
We first observe that, by the boundedness of $\{v_{n_k}\}$ in $L^2(\rn)$,
there exists $C\ge0$ such that
\[
\mathrm{II}_k\le R^{-2}\int_{\rn}|v_{n_k}(x)|^2dx\le CR^{-2}.
\] 
By choosing $R>0$ large, we can thus make $\mathrm{II}_k$ as small as we want, uniformly in $k$. 

Next, we fix $R>0$ arbitrary and prove that $\lim_{k\to\infty}\mathrm{I}_k=0$. 
This will complete the proof of \eqref{zerolimit}. 
We argue by contradiction and suppose there exist $\delta>0$ 
and a subsequence of $\{v_{n_k}\}$ (still denoted $\{v_{n_k}\}$) such that
\begin{equation*}\label{deltaR}
\liminf_{k\rightarrow\infty}\int_{|x|<R}\frac{|v_{n_k}|^2}{|x|^2}dx\ge\delta.
\end{equation*}
Note that this inequality will still hold if we increase the value of $R$ or pass to a further subsequence.
Let $a>0$ and $\ffi\in C_0^\infty([0,\infty))$ be such that $\ffi(r)=1$ for $0\le r \le R$,
$\ffi(r)=0$ for $r\ge R+a$, $0\le\ffi\le1$ and $\|\ffi'\|_{L^\infty}\le 1/a$. 
Then define
\[
u_k(x)=\ffi(|x|) v_{n_k}(x) \quad\text{and}\quad w_k(x)=(1-\ffi(|x|))v_{n_k}(x).
\]
We clearly have $|u_k|,|w_k|\le|v_{n_k}|$, and $\{u_k\},\{w_k\}\subset H^1(\rn)$ with  
\[
|\nabla u_k|^2 \le 2\left((\ffi')^2|v_{n_k}|^2+\ffi^2|\nabla v_{n_k}|^2\right)
\le 2\left(a^{-2}|v_{n_k}|^2+|\nabla v_{n_k}|^2\right)
\]
and
\[
|\nabla w_k|^2 \le 2\left((\ffi')^2|v_{n_k}|^2+(1-\ffi)^2|\nabla v_{n_k}|^2\right)
\le 2\left(a^{-2}|v_{n_k}|^2+|\nabla v_{n_k}|^2\right).
\]
Furthermore,
\begin{equation*}\label{deltaRuk}
\liminf_{k\rightarrow\infty}\intrn\frac{|u_k|^2}{|x|^2}dx\ge
\liminf_{k\rightarrow\infty}\int_{|x|<R}\frac{|v_{n_k}|^2}{|x|^2}dx\ge\delta,
\end{equation*}
and \eqref{hardy} yields
\begin{equation}\label{crucial}
\liminf_{k\rightarrow\infty}H(u_k)=\liminf_{k\rightarrow\infty}
\left(\intrn |\nabla u_k|^2dx-c\intrn\frac{|u_k|^2}{|x|^2}dx\right) \ge (c_*-c)\delta>0.
\end{equation}
Using the sequences $\{u_k\},\{w_k\}$, we can now rewrite $H(v_{n_k})$ as
\begin{align*}
H(v_{n_k})
&=H(u_k)+H(w_k) \\
&\phantom{=}+\int_{R<|x|<R+a} 
\Big[|\nabla v_{n_k}|^2-|\nabla u_k|^2-|\nabla w_k|^2
-\frac{c}{|x|^2}\left(|v_{n_k}|^2-|u_k|^2-|w_k|^2\right)\Big]dx.
\end{align*}
We will show that there exist $R>0$ as large as we want and $a>0$
such that, up to a further subsequence,
\begin{equation}\label{est0}
\limsup_{k\to\infty}\left|\int_{R<|x|<R+a} 
\Big[|\nabla v_{n_k}|^2-|\nabla u_k|^2-|\nabla w_k|^2
-\frac{c}{|x|^2}\left(|v_{n_k}|^2-|u_k|^2-|w_k|^2\right)\Big]dx\right|\le\frac{(c_*-c)\delta}{2}.
\end{equation}
Using \eqref{crucial}, this will imply
\begin{equation}\label{e1}
\liminf_{k\rightarrow\infty}H(v_{n_k})
\ge\liminf_{k\rightarrow\infty}H(w_k)+\frac{(c_*-c)\delta}{2}>\liminf_{k\rightarrow\infty}H(w_k),
\end{equation}
which will lead to the desired contradiction.
To prove \eqref{est0}, we use the above properties of $u_k$ and $w_k$ to observe that
\begin{multline}\label{est}
\left|\int_{R<|x|<R+a}\Big[|\nabla v_{n_k}|^2-|\nabla u_k|^2-|\nabla w_k|^2
-\frac{c}{|x|^2}\left(|v_{n_k}|^2-|u_k|^2-|w_k|^2\right)\Big]dx\right| \\
\le 5\int_{R<|x|<R+a}|\nabla v_{n_k}|^2dx+4a^{-2}\int_{R<|x|<R+a}|v_{n_k}|^2dx
+\frac{3c}{R^2}\int_{R<|x|<R+a}|v_{n_k}|^2dx.
\end{multline}
First, by choosing $R>0$ sufficiently large, we will have
\begin{equation}\label{term3}
\frac{3c}{R^2}\int_{R<|x|<R+a}|v_{n_k}|^2dx  
\le \frac{3c}{R^2} \|v_{n_k}\|_{L^2}^2 \le \frac{(c_*-c)\delta}{4} \quad\text{for all} \ k\in\N.
\end{equation}
The second term in the right-hand side of \eqref{est} is handled using Step 1 with $|y_k|\to\infty$: 
for any fixed $R,a>0$, we have that
\begin{equation}\label{term2}
\lim_{k\to\infty} \int_{R<|x|<R+a}|v_{n_k}|^2dx=0.
\end{equation}
As for the first term, we claim that we can choose $R>0$ as large as we want and $a>0$ such that,
up to a subsequence,
\begin{equation}\label{term1}
5\int_{R<|x|<R+a}|\nabla v_{n_k}|^2dx\le \frac{(c_*-c)\delta}{4} \quad\text{for all} \ k\in\N.
\end{equation}
Estimates \eqref{term3}--\eqref{term1} prove \eqref{est0}, and hence \eqref{e1}.
If the above claim is not true, then using a diagonal argument we can construct a 
subsequence $\{v_{n_{k_l}}\}$, and a sequence of pairwise disjoint annuli
\[
A_j=\{x\in\rn : R_j<|x|<R_j+a_j\} \quad\text{with}\quad R_{j+1}>R_j+a_j,
\]
such that, for all $j\in\{1,\dots,l\}$,
\[
\int_{A_j}|\nabla v_{n_{k_l}}|^2dx\ge \bar \delta:= \frac{(c_*-c)\delta}{20}>0 \quad\text{for all} \ n_{k_l}\in\N.
\]
But then
\[
\intrn|\nabla v_{n_{k_l}}|^2dx \ge \sum_{j\in\N}\int_{A_j}|\nabla v_{n_{k_l}}|^2dx
\ge \sum_{j=1}^l\int_{A_j}|\nabla v_{n_{k_l}}|^2dx \ge l\bar\delta,
\]
and it follows that
\[
\lim_{l\to\infty}\intrn|\nabla v_{n_{k_l}}|^2dx=+\infty.
\]
In view of \eqref{hardy2} this contradicts \eqref{a}, and completes the proof of \eqref{e1}.

We will now combine \eqref{e1} with the Gagliardo--Nirenberg inequality \eqref{GN} and
the assumption on the energy in \eqref{a} to reach a contradiction.
Using Step 1 with $|y_k|\to\infty$, we have that
$u_k\to0$ in $L^{p+1}(\rn)$.
Hence, by \eqref{a},
\begin{align*}
\limsup_{k\rightarrow\infty}H(v_{n_k})\le\frac{2}{p+1}\liminf_{k\rightarrow\infty}\int_{\rn}|v_{n_k}|^{p+1}dx
=\frac{2}{p+1}\liminf_{k\rightarrow\infty}\int_{\rn}|w_k|^{p+1}dx,
\end{align*}
and so, by \eqref{e1},
\begin{align} \label{e4}
\liminf_{k\rightarrow\infty}H(w_k)<\frac{2}{p+1}\liminf_{k\rightarrow\infty}\int_{\rn}|w_k|^{p+1}dx.
\end{align}
Using again Step 1 with $|y_k|\to\infty$, we also have that $u_k\to0$ in $L^2(\rn)$,
which implies
\[
M_\mathrm{gs}=\lim_{k\rightarrow\infty}\|v_{n_k}\|_{L^2}
=\lim_{k\rightarrow\infty}\|w_k\|_{L^2}.
\]
Therefore, \eqref{GN} applied to $w_k$ yields 
\[
\frac{2}{p+1}\liminf_{k\rightarrow\infty}\int_{\rn}|w_k|^{p+1}dx
\le\liminf_{k\rightarrow\infty}\left(\frac{\|w_k\|_{L^2}}{M_\mathrm{gs}}\right)^{\frac4N}H(w_k)
= \liminf_{k\rightarrow\infty}H(w_k),
\]
which, together with \eqref{e4}, leads to the contradiction
\[
\liminf_{k\rightarrow\infty}H(w_k)<\liminf_{k\rightarrow\infty}H(w_k).
\]
This shows that $\lim_{k\to\infty}\mathrm{I}_k=0$ and completes the proof of \eqref{zerolimit}.

We now have
\begin{align*}
\liminf_{k\to\infty} E(v_{n_k})
&=\liminf_{k\to\infty} \frac12 \|\nabla v_{n_k}\|_{L^2}^2
-\lim_{k\to\infty} \frac{1}{p+1} \| v_{n_k}\|_{L^{p+1}}^{p+1}\\
&=\liminf_{k\to\infty} \frac12 \|\nabla \wt{v}_{n_k}\|_{L^2}^2
-\lim_{k\to\infty} \frac{1}{p+1} \| \wt{v}_{n_k}\|_{L^{p+1}}^{p+1}\\
&\ge \frac12 \|\nabla v\|_{L^2}^2 - \frac{1}{p+1} \|v\|_{L^{p+1}}^{p+1}>E(v)\ge 0,
\end{align*}
where the last inequality follows from \eqref{posE} with $\|v\|_{L^2}=M_\mathrm{gs}$.
This contradicts \eqref{a} and shows that the sequence $\{y_k\}$ must indeed be bounded in $\rn$.

\medskip
\noindent
{\em Step~3: Conclusion.} Suppose, without loss of generality, that $y_k\to y^*\in\rn$.
Then
\[
v_{n_k}\to v^*=v(\cdot+y^*) \quad \text{in} \quad L^q(\rn), \quad \forall q\in[2,2^*),
\]
and $v_{n_k}\wto v^*$ weakly in $H^1(\rn)$.
Since the Hardy functional $H$ is weakly lower semi-continuous \cite{montefusco},
and $p+1\in(2,2^*)$, it now follows by \eqref{posE} that
\[
\begin{split}
0\le E(v^*)=\frac{1}{2}H(v^*)-\frac{1}{p+1}\int_{\rn}|v^*|^{p+1}dx
&\le \frac{1}{2}\limsup_{k\rightarrow\infty}H(v_{n_k})-\frac{1}{p+1}\int_{\rn}|v^*|^{p+1}dx\\
&\phantom{\le}=\limsup_{k\rightarrow\infty}E(v_{n_k})\le 0,
\end{split}
\]
so in fact
\begin{equation}\label{convergence}
E(v^*)=0 \quad\text{and}\quad H(v^*)=\lim_{k\rightarrow\infty}H(v_{n_k}).
\end{equation}
Since $\|v^*\|_{L^2}=M_\mathrm{gs}$ and $E(v^*)=0$, Proposition \ref{charact.prop} yields
$v^*=e^{i\gamma_0}Q_{\lambda_0}$ for some $\gamma_0\in\real$, $\lambda_0>0$, 
and some $Q\in \mathcal{G}$. 
Finally, by \eqref{convergence} and \eqref{hardy2}, 
we also have that $\|v_{n_k}\|_{H^1}\rightarrow \|v^*\|_{H^1}$.
Hence $\{v_{n_k}\}$ converges strongly to $v^*$ in $H^1(\rn)$, which completes the proof.
\end{proof}


\subsection{Virial identities}

We now prove virial identities for \eqref{inls}. 
As they can be useful in more general contexts, we state them for any $1<p<1+\frac{4}{N-2}$.
We let
\[
\Sigma :=\{u \in H^1(\rn) : xu \in L^2(\rn) \},
\]
and for $u(t)\in\Sigma$, we introduce
\[
\Gamma(t):=\int_{\rn} |x|^2|u(t,x)|^2dx.
\]

\begin{lemma}\label{virial.lemma}
Let $u$ be a solution of \eqref{inls} on $[0,T)$, such that $u(t)\in \Sigma$ for all $t\in [0,T)$. 
Then for all $t\in[0,T)$ the following identities hold:
\begin{equation} \label{virialone}
\Gamma'(t)=4\im\int_{\rn} \overline{u}(t,x)(\nabla u(t,x)\cdot x)dx
\end{equation}
and
\begin{equation}\label{virialtwo}
\Gamma''(t)=16E(u)+\frac{4}{p+1}(N-Np+4)\int_{\rn} |u(t,x)|^{p+1}dx.
\end{equation}
\end{lemma}

\begin{proof} By regularization we may assume that 
$u$ is smooth enough for the following calculation. Since $u$ satisfies (\ref{inls}), we find
\[
\begin{split}
\Gamma'(t)&=2\re\int|x|^2\bar{u}\partial_tu=2\re\int|x|^2\bar{u}i(\Delta u+\frac{c}{|x|^2}u+|u|^{p-1}u)\\
&=-2\im\int|x|^2(\bar{u}\Delta u+\frac{c}{|x|^2}|u|^2+|u|^{p+1})=-2\im\int|x|^2\bar{u}\Delta u.
\end{split}
\]
Integrating by parts and using that $\nabla |x|^2=2x$, we obtain
\[
\Gamma'(t)=2\im\int\nabla u(\bar{u}\nabla|x|^2+|x|^2\nabla\bar{u})=4\im\int\bar{u}(\nabla u\cdot x).
\]
Using again integration by parts, we get
\[
\begin{split}
\Gamma''(t)&=4\im\int\partial_t\bar{u}(x\cdot\nabla u)+\bar{u}(x\cdot\nabla\partial_tu)
=4\im\int\partial_tu[-x\cdot\nabla\bar{u}-\nabla\cdot(\bar{u}x)]\\
&=4\im\int\partial_tu[-2x\cdot\nabla\bar{u}-\bar{u}\nabla\cdot x]
=-8\im\int\partial_tu(x\cdot\nabla\bar{u})-4N\im\int\partial_tu\bar{u}.
\end{split}
\]
We use (\ref{inls}) to compute the last two terms. We first find

\begin{align}
-4N\im\int\partial_tu\bar{u}&=-4N\re\int\bar{u}(\Delta u+\frac{c}{|x|^2}u+|u|^{p-1}u)\notag
\\
&=4N\int{|\nabla u|^2}-4Nc\int\frac{|u|^2}{|x|^2}-4N\int|u|^{p+1}. \label{viriala}
\end{align}
Similarly,
\[
\begin{split}
-8\im\int\partial_tu(x\cdot \nabla \bar{u})&=
-8\re\int(x\cdot\nabla\bar{u})(\Delta u+\frac{c}{|x|^2}u+|u|^{p-1}u)
\\
&=-8\re\int\Delta u(x\cdot\nabla\bar{u})-8c\int|x|^{-2}x\cdot\re(u\nabla\bar{u})
-8\int x\cdot|u|^{p-1}\re(u\nabla\bar{u})
\\
&\equiv A+B+C.
\end{split}
\]
Since $\partial_{x_k}(|\partial_{x_j}u|^2) = 2\re(\partial_{x_j}u\partial_{x_j}\partial_{x_k}\bar u)$
for $1\le j,k\le N$, we find
\begin{align}
A &= -8\sum_{j,k} \re\int \partial_{x_j}^2u\, x_k\partial_{x_k}\bar u
= 8\sum_{j,k} \re\int \partial_{x_j}u (\delta_{j,k}\partial_{x_k}\bar u + x_k\partial_{x_j}\partial_{x_k}\bar u) \notag \\
&= 8\sum_j \int |\partial_{x_j}u|^2 +4\sum_{j,k} \int x_k\partial_{x_k}(|\partial_{x_j}u|^2) \notag \\
&= 8\sum_j \int |\partial_{x_j}u|^2 -4\sum_{j,k} \int |\partial_{x_j}u|^2 \label{virialpart3}
= (8-4N)\int |\nabla u|^2, 
\end{align}
where $\delta_{j,k}=1$ for $j=k$ and $0$ otherwise.

Since $\nabla(|u|^{p+1})=(p+1)|u|^{p-1}\re(u\nabla\bar{u})$, we obtain by integration by parts
\begin{align}
C=\frac{8}{p+1}\int|u|^{p+1}\nabla\cdot x=\frac{8N}{p+1}\int|u|^{p+1}.\label{virialc}
\end{align}
Similarly, we get
\begin{align}
B&=-8c\int \nabla(|u|^2)\cdot|x|^{-2}x=+4c\int \nabla(|x|^{-2}x)|u|^2=+4c\int|u|^2(\nabla|x|^{-2}\cdot x+|x|^{-2}\nabla \cdot x) \notag
\\
&\phantom{=}+4c\int|u|^2(-2|x|^{-2}+N|x|^{-2}).\label{viriald}
\end{align}
Finally, putting together (\ref{viriala}), (\ref{virialpart3}), (\ref{virialc}) and (\ref{viriald}), we obtain
\[
\begin{split}
\Gamma''(t)&=8\int|\nabla u|^2-8c\int\frac{|u|^2}{|x|^2}+\frac{4}{p+1}(N-Np)\int|u|^{p+1}
\\
&=16E(u)+\frac{4}{p+1}(N-Np+4)\int|u|^{p+1},
\end{split}
\]
which concludes the proof. 
\end{proof}

\begin{remark} In the $L^2$ critical case, $p=1+\frac{4}{N}$, 
(\ref{virialtwo}) reduces to
\begin{equation}\label{virialthree}
\Gamma''(t)=16E(u_0).
\end{equation}
\end{remark}


\subsection{Classification}
To prove our main result, we shall deduce information about $u(t)$ for $t<T$ from the
blow-up at $t=T$. In particular, using a result of Banica \cite{banica}, we will show
that $u(t)\in\Sigma$ for all $t<T$. This will allow us to apply the virial identities, which will
lead to the conclusion.

For $u\in H^1(\rn)$, $\theta \in C_0^\infty(\rn,\R)$ and $s\in \R$, 
we have that $\nabla(ue^{is\theta})=(\nabla u + is\nabla\theta)e^{is\theta}$, and so
\[
|\nabla(ue^{is\theta})|^2=|\nabla u|^2+2s\nabla\theta\cdot\im(\bar{u}\nabla u)+s^2|\nabla\theta|^2|u|^2.
\]
An easy calculation then yields
\begin{equation}\label{valeria}
E(ue^{is\theta})=E(u)+s\int_{\rn}\nabla \theta\cdot \im(\overline{u}\nabla u)dx 
+ \frac{s^2}{2}\int_{\rn}|\nabla\theta|^2|u|^2dx.
\end{equation}

\begin{lemma}[Lemma~2.1 in \cite{banica}]\label{banica.lem}
Let $u\in H^1(\rn)$ such that $\|u\|_{L^2}=\|Q\|_{L^2}$. Then for all $\theta \in C_0^\infty(\rn)$ we have that
\[
\left|\int_{\rn} \nabla \theta\cdot Im(\overline{u}\nabla u)dx\right|
\le \sqrt{2E(u)}\left(\int_{\rn}|\nabla \theta|^2|u|^2\right)^{1/2}.
\]
\end{lemma}

We can now prove our main result.

\begin{proofof}~{\it Theorem~\ref{main.thm}}.
Let $u$ be a solution of \eqref{inls} such that $\|u\|_{L^2}=M_\mathrm{gs}$ 
and $\lim_{t\uparrow T} H(u) =+\infty$.

\medskip
\noindent
{\em Step~1.} 
Let $\{t_n\}_{n\in\nat} \subset \real$ be a sequence of times such that $t_n\uparrow T$ 
as $n\rightarrow\infty$. We set
\[
u_n=u(t_n), \quad \lambda_n^2=H(u_n), \quad v_n(x)=\lambda_n^{-N/2}u_n(\lambda_n^{-1}x).
\]
We note that $\lambda_n\rightarrow\infty$ as $n\rightarrow\infty$, 
and $\|v_n\|_{L^2}=\|u_n\|_{L^2}=\|u_0\|_{L^2}=M_\mathrm{gs}$ by the $L^2$ scaling.
Furthermore,
\[
\begin{split}
H(v_n)=\lambda_n^{-2}H(u_n)=1.
\end{split}
\]
On the other hand, by conservation of the energy,
\[
\begin{split}
E(v_n)&=\frac12\lambda_n^{-2}H(u_n)-\frac{1}{p+1}\lambda_n^{-2}\int_{\rn} |u_n(y)|^{p+1}dy\\
&=\lambda_n^{-2} E(u_n)=\lambda_n^{-2}E(u_0)\to 0 \quad \text{as} \ n\to\infty.
\end{split}
\]
Hence, by Proposition \ref{compact.prop}, there exist $Q\in\mathcal{G}$, $\gamma_0\in\real$ 
and $\lambda_0>0$ such that, up to a subsequence,
\begin{equation}\label{b}
\lim_{n\rightarrow\infty} \|v_n-e^{i\gamma_0}Q_{\lambda_0}\|_{H^1}=0,
\end{equation}
where $Q_{\lambda_0}(x)=\lambda_0^{N/2}Q(\lambda_0 x)$.

\medskip
\noindent
{\em Step~2.} 
We now prove that $u_n$ concentrates all of its mass at $x=0$ as $n\rightarrow\infty$. 
We show that, in the sense of distributions,
\[
|u_n|^2\rightarrow M_\mathrm{gs}^2\delta_0.
\]
Indeed, for any $\varphi\in C^\infty_0(\rn)$, using the change of variables $y=\lambda_n^{-1}x$
we find that
\[
\begin{split}
\int_{\rn}|u_n(y)|^2\varphi(y)dy
&=\int_{\rn}|u_n(\lambda_n^{-1}x)|^2\varphi(\lambda^{-1}_nx)\lambda_n^{-N}dx
=\int_{\rn}|v_n(x)|^2\varphi(\lambda_n^{-1}x)dx\\
&=\int_{\rn}(|v_n(x)|^2-|Q(x)|^2)\varphi(\lambda_n^{-1}x)dx+\int_{\rn} |Q(x)|^2\varphi(0)dx\\
&\phantom{\le} \ +\int_{\rn}|Q(x)|^2(\varphi(\lambda_n^{-1}x)-\varphi(0))dx.
\end{split}
\]
Thus, we obtain
\[
\begin{split}
\Big|\int_{\rn}|u_n(y)|^2\varphi(y)dy -M_\mathrm{gs}^2\varphi(0)\Big|
&\le \|\varphi\|_{L^\infty}\int_{\rn}||v_n(x)|^2-|Q(x)|^2|dx\\
&\phantom{\le} \ +\int_{\rn}|Q(x)|^2|\varphi(\lambda_n^{-1}x)-\varphi(0)|dx.
\end{split}
\]
We conclude this step by noticing that $|v_n|^2$ converges strongly to 
$|Q|^2$ in $L^1(\rn)$ from (\ref{b}), so the first integral vanishes as $n\rightarrow\infty$. 
Since $\lambda_n\rightarrow\infty$ the second integral also converges to zero 
by the dominated convergence theorem. Hence, we obtain
\begin{equation} \label{c}
\int_{\rn}|u_n(y)|^2\varphi(y)dy\rightarrow M_\mathrm{gs}^2\varphi(0).
\end{equation}

\medskip
\noindent
{\em Step~3.} 
We now show that $u(t)\in\Sigma$ for all $t\in[0,T)$. 
Let $\phi\in C^\infty_0(\rn)$ be a radial and non-negative function such that 
$\phi(x)=|x|^2$ for $|x|\le 1$. It is easy to see that there exists a constant $C>0$ such that
\[
|\nabla \phi (x) |^2\le C\phi(x)
\]
for all $x\in\rn$.
Indeed, since $\phi$ is radially symmetric we may choose $f\in C(\real,\real_+) $, such that $\phi(|x|)=f(|x|)$. By Taylor's formula, for all $r\in\real$ and $h\in\real$ there exists $y\in[r,r+h]$, such that
\[
 0\le f(r+h)=f(r)+f'(r)h+\frac{f''(y)}{2}h^2\le f(r)+f'(r)h+C'h^2,
\]
where $C'=1+\max_{r\in\real}\frac{|f'(r)|}{2}>0$. Since the right-hand side is a positive    
quadratic polynomial, we must have that $|f'(r)|^2-4C'f(r)\le 0$, thus $|f'(r)|^2\le Cf(r)$, where $C=4C'>0$.

For all $R>0$, we define $\phi_R(x)=R^2\phi(x/R)$, and for all $t\in[0,T)$,
\[
\Gamma_R(t)=\int_{\rn}\phi_R(x)|u(t,x)|^2dx.
\]
A direct calculation, similar to the proof of Lemma~\ref{virial.lemma}, yields that
\[
\Gamma_R'(t)=2\text{Re}\int\phi_R\bar{u}\partial_tu 
=-2\text{Im}\int\phi_R\bar{u}(\Delta u+\frac{c}{|x|^2}u+|u|^{p-1}u)
=2\int\nabla\phi_R\cdot\text{Im}(\bar{u}\nabla u).
\]
Since $\|u\|_{L^2}=\|Q\|_{L^2}$, we can apply Lemma \ref{banica.lem} 
and the inequality $|\nabla\phi_R|^2\le C\phi_R$, to get 
\[
|\Gamma'_R(t)|\le 2\sqrt{2E(u)}\Big(\int|\nabla\phi_R|^2|u|^2\Big)^{1/2}
\le C\sqrt{E(u_0)}\sqrt{\Gamma_R(t)}.
\]
By integrating between a fixed $t\in[0,T)$ and $t_n$, we obtain,
\[
|\sqrt{\Gamma_R(t)}-\sqrt{\Gamma_R(t_n)}|\le C|t-t_n|.
\]
Applying (\ref{c}) in Step 2, we get
\[
\Gamma_R(t_n)=\int_{\rn} |u_n(x)|^2\phi_R(x)dx\rightarrow M_\mathrm{gs}^2\phi_R(0)=0.
\]
Thus, letting $n\to\infty$ in the last inequality, we obtain, for all $t\in [0,T)$ and all $R>0$,
\[
\Gamma_R(t) \le C(T-t)^2.
\]
Since the right-hand side of the last expression is independent of $R$,
letting $R\to\infty$ yields, for all $t\in [0,T)$,
\[
u(t)\in\Sigma \quad\text{and}\quad 0\le \Gamma(t) \le C(T-t)^2.
\]
From this estimate, we can extend by continuity $\Gamma(t)$ at $t=T$ by setting $\Gamma(T)=0$,
from which we also obtain $\Gamma'(T)=0$.
Moreover, since $u(t)\in\Sigma$ and $u$ is a solution of~\eqref{inls}, 
we may apply Lemma~\ref{virial.lemma}, and by (\ref{virialthree}) 
we obtain $\Gamma''(t)=16E(u_0)$, which finally gives, for all $t\in [0,T)$,
\[
\Gamma(t) = 8E(u_0)(T-t)^2.
\]
Letting $t=0$, we find, using identity (\ref{virialone})  
\[
\Gamma(0)=\int |x|^2|u_0|^2= 8E(u_0)T^2
\quad \text{and}\quad \Gamma'(0)= 4\int x\cdot\im(\overline{u_0}\nabla u_0) = -16E(u_0)T.
\]

\medskip
\noindent
{\em Step 4.} 
We apply identity \eqref{valeria} with $u_0$, $s=\frac{1}{2T}$, and $\theta(x)=\frac{|x|^2}{2}$ to get 
\[
\begin{split}
E(u_0e^{i\frac{|x|^2}{4T}})&=E(u_0)+\frac{1}{2T}\int x\cdot\text{Im}(\bar{u_0}\nabla u_0)
+\frac{1}{8T^2}\int|x|^2|u_0|^2\\
&=E(u_0)+\frac{1}{2T}(-4E(u_0))+\frac{1}{8T^2}(8E(u_0)T^2)=0.
\end{split}
\]
Note that this calculation justifies, {\em a posteriori}, the application of~\eqref{valeria}
with $\theta(x)=\frac{|x|^2}{2}\not\in C_0^{\infty}(\rn)$.
Since $\|u_0e^{i\frac{|x|^2}{4T}}\|_{L^2}=M_\mathrm{gs}$ and $E(u_0e^{i\frac{|x|^2}{4T}})=0$, 
we can deduce from Proposition \ref{charact.prop} that there exist $\lambda_1>0$, 
$\gamma_1\in\real$, and $\wt{Q}\in\mathcal{G}$ such that
\[
u_0(x)=e^{i\gamma_1}e^{-i\frac{|x|^2}{4T}}\lambda_1^{N/2}\wt{Q}(\lambda_1x).
\]
Finally, we use the pseudo-conformal transformation. We define $\tilde\lambda_0=\lambda_1T>0$ and 
$\tilde\gamma_0=\gamma_1-\lambda_1^2T\in\real$, and write $u_0$ as
\[
u_0(x)=e^{i\tilde\gamma_0}e^{i\frac{\tilde\lambda^2_0}{T}}
e^{-i\frac{|x|^2}{4T}}\Big(\frac{\tilde\lambda_0}{T}\Big)^{N/2}\wt{Q}\Big(\frac{\tilde\lambda_0x}{T}\Big).
\]
Thus, $u_0=S_{\wt{Q},T,\tilde\lambda_0,\tilde\gamma_0}(0)$, where 
$S_{\wt{Q},T,\tilde\lambda_0,\tilde\gamma_0}$ 
is defined by (\ref{S2.def}). By invoking uniqueness of the solution of (\ref{inls}), 
we find that $u(t)=S_{\wt{Q},T,\tilde\lambda_0,\tilde\gamma_0}(t)$ for all $t\in[0,T)$, 
which concludes the proof.  
\end{proofof}


\end{document}